\documentclass[reqno,12pt]{amsart} 
\usepackage{amsmath}
\usepackage{amssymb}
\usepackage{amsthm}
\newcommand\NoBlackBoxes{\global\overfullrule0pt}
\NoBlackBoxes
\theoremstyle{plain} 
\newtheorem{theorem}{Theorem} 
\newtheorem{lemma}[theorem]{Lemma}
\newtheorem{proposition}[theorem]{Proposition}

\newtheorem{corollary}[theorem]{Corollary}
\def\4{\kern1pt}

\def\6{\vphantom0}

\def\8{\kern-10pt}
\def\7#1{_{(#1)}}

\theoremstyle{definition}
\newtheorem*{assumption (H1)}{Assumption (H1)}
\newtheorem*{assumption (H2)}{Assumption (H2)}

\theoremstyle{remark}
\newtheorem{remark}[theorem]{Remark}

\numberwithin{equation}{section}
\numberwithin{theorem}{section}
\makeatletter
\let\serieslogo@\relax
\let\@setcopyright\relax

\def\speciallabelmark#1{\def\@currentlabel{#1}}
\makeatother

\newcommand{\Cal}{\mathcal}

\newcommand{\beqa}{\begin{eqnarray}}
\newcommand{\beqan}{\begin{eqnarray*}}
\newcommand{\eeqa}{\end{eqnarray}}
\newcommand{\eeqan}{\end{eqnarray*}}
\def\beq#1\eeq{\begin{equation}#1\end{equation}}

\begin{document}

\def\ffrac#1#2{\raise.5pt\hbox{\small$\4\displaystyle\frac{\,#1\,}{\,#2\,}\4$}}
\def\ovln#1{\,{\overline{\!#1}}}
\def\ve{\varepsilon}
\def\kar{\beta_r}

\title{Rate of Convergence in the entropic free CLT}

\author{G. P. Chistyakov$^{1,2}$}
\thanks{1) Faculty of Mathematics , 
University of Bielefeld, Germany.} 
\thanks{2) Research supported by SFB 701.}
\address
{Gennadii Chistyakov \newline
Fakult\"at f\"ur Mathematik\newline
Universit\"at Bielefeld\newline
Postfach 100131\newline
33501 Bielefeld \newline
Germany}
\email {chistyak@math.uni-bielefeld.de} 

\author{F. G\"otze$^{1,2}$}
\address
{Friedrich G\"otze\newline
Fakult\"at f\"ur Mathematik\newline
Universit\"at Bielefeld\newline
Postfach 100131\newline
33501 Bielefeld \newline
Germany}
\email {goetze@math.uni-bielefeld.de}

\date{December, 2011}


\maketitle
\markboth{ G. P. Chistyakov and F. G\"otze}{Rate of convergence}

\begin{abstract}
We prove an expansion for densities in the free CLT and apply this result to an expansion
in the entropic free central limit theorem assuming a~moment condition of order 8 for the free summands.  
\end{abstract}

\section{Introduction}
Free convolutions were introduced
by D. Voiculescu~\cite{Vo:1986}, \cite{Vo:1987} and have been studied 
intensively in order to understand non commutative probability.
The~key concept here is the~notion of freeness,
which can be interpreted as a~kind of independence for
non commutative random variables. As in~classical probability theory where
the~concept of independence gives rise to the~classical convolution, 
the~concept of freeness leads to a~binary operation on the~probability measures, 
the~free convolution. Many classical results 
in the~theory of addition of independent random variables have their
counterparts in Free Probability, such as the~Law of Large Numbers,
the~Central Limit Theorem, the~L\'evy-Khintchine formula and others.
We refer to Voiculescu, Dykema and Nica \cite{Vo:1992} and Hiai 
and Petz~\cite{HiPe:2000} for an~introduction 
to these topics.

In this paper we 
obtain an~analogue of Esseen's expansion for a~density of normalized sums of free identically distributed
random variables under moment assumptions on free summands. Using this expansion we establish
the rate of convergence of the free entropy of normalized sums of free identically distributed
random variables.

The~paper is organized as follows. In Section~2 we formulate and
discuss the main results of the~paper. In Section~3 and 4 we state 
auxiliary results. 
In Sections~5 we discuss the passage to probability measures with bounded supports.
In Sections~6 we obtain a~local asymptotic expansion for a~density
in the~CLT for free identically distributed random variables. 
In Sections~7 we study a behaviour of subordination functions in the free CLT for truncated free summands.
In Sections~8 we discuss the closeness of subordination functions in the free CLT for bounded and unbounded  free
random variables.
In Sections~9 we  prove the rate of convergence for a~density in the free CLT in $L_1(-\infty,+\infty)$.
In Sections~10 we prove the rate of convergence for the free entropy of normalized sums of free identically distributed
random variables.

\section{Results}
Denote by $\mathcal M$ the~family of all Borel probability measures
defined on the~real line $\mathbb R$.  
Let $\mu\boxplus\nu$ be
the~free (additive) convolution of $\mu$ and $\nu$ introduced by 
Voiculescu~\cite{Vo:1986} for compactly supported measures. 
Free convolution was extended by
Maassen~\cite{Ma:1992} to measures with finite variance and by Bercovici
and Voiculescu~\cite{BeVo:1993} to the~class $\mathcal M$.
Thus, $\mu\boxplus\nu=\mathcal L(X+Y)$,
where $X$ and $Y$ are free random variables such that $\mu=\mathcal L(X)$ and $\nu=\mathcal L(Y)$.

Henceforth $X,X_1,X_2,\dots$ stands for a~sequence of identically distributed random variables with
distribution $\mu=\mathcal L(X)$. Define 
$m_k(\mu):=\int_{\mathbb R}u^k\,\mu(du)$,
where $k=0,1,\dots$.

The~classical CLT says that if $X_1,X_2,\dots$ are independent
and identically distributed random variables with a~probability 
distribution $\mu$ such that $m_1(\mu)=0$ and $m_2(\mu)=1$, 
then the~distribution function $F_n(x)$ of 
\begin{equation}\label{2.1}
Y_n:=\frac{X_1+X_2+\dots +X_n}{\sqrt n}
\end{equation}
tends to the~standard Gaussian law $\Phi(x)$ as $n\to\infty$ uniformly
in $x$.

A~free analogue of this classical result was proved by Voiculescu~\cite{Vo:1985} for bounded
free random variables and 
later generalized by Maassen~\cite {Ma:1992} to unbounded random variables.
Other generalizations can be found in \cite{BeVo:1995}, \cite{BeP:1999}, \cite{ChG:2005a}, 
\cite {Ka:2007}--\cite {Ka:2007b}, \cite{P:1996}, \cite{Vo:2000}, \cite{Wa:2010}. 
When the~assumption of independence is 
replaced by the~freeness of the~non commutative random variables
$X_1,X_2,\dots,X_n$, the~limit distribution function of (\ref{2.1}) is 
the~semicircle law $w(x)$, i.e., the~distribution function 
with the~density $p_w(x):=\frac 1{2\pi}\sqrt{(4-x^2)_+}, \,x\in\mathbb R$, 
where $a_+:=\max\{a,0\}$ for $a\in\mathbb R$. Denote by $\mu_w$ the~probability measure with
the~distribution function $w(x)$. We denote as well by $\mu_n$ the probability measure with
the distribution function $F_n(x)$.

It was proved in~\cite{BelBer:2004} that if the distribution $\mu$ of $X$ is not a Dirac measure,
then in the free case $F_n(x)$ is Lebesgue absolutely continuous when $n\ge n_1=n_1(\mu)$ is sufficiently large. 
Denote by $p_n(x)$ the density of $F_n(x)$.

In the sequel we denote by $c(\mu),c_1(\mu),c_2(\mu),\dots$
positive constants depending on $\mu$ only. By $c(\mu)$ we denote generic constants in different 
(or even in the same) formulae. The symbols $c_1(\mu),c_2(\mu),\dots$ will denote explicit constants.

Wang~\cite{Wa:2010} proved that under the condition $m_2(\mu)<\infty$ 
the density $p_n(x)$ of $F_n(x)$ is continuous for sufficiently large $n$ and  
\begin{equation}\label{asden3}
p_n(x)\le c(\mu),\quad x\in\mathbb R.
\end{equation}

Assume that $m_4(\mu)<\infty,m_1(\mu)=0,m_2(\mu)=1$ and denote 
\begin{equation}\label{2.3****}
a_n:=\frac{m_3(\mu)}{\sqrt n},\quad b_n:=\frac{m_4(\mu)-m_3^2(\mu)-1}n,\quad
d_n:=\frac{m_4(\mu)-m_3^2(\mu)}n,\quad n\in\mathbb N.
\end{equation}
Denote as well $e_n:=(1-b_n)/\sqrt{1-d_n}$ and by $I_n$ the interval of the form
\begin{equation}\label{asden1}
I_n:=\Big\{x\in\mathbb R:|x-a_n|\le \frac 2{e_n}-c_1(\mu)n^{-6/5}\Big\} 
\end{equation}
with some constant $c_1(\mu)$. 

We have derived an asymptotic expansion of $p_n(x)$ for bounded 
free random variables $X_1,X_2,\dots$ in the paper~\cite{ChG:2011}. Improving the methods of this paper we
obtain an asymptotic expansion of $p_n(x)$ for the case $m_8(\mu)<\infty$.
\begin{theorem}\label{th7}
Let $m_8(\mu)<\infty$ and $m_1(\mu)=0,\,m_2(\mu)=1$. Then, for $n\ge n_1$,
\begin{equation}\label{asden}
p_n(x+a_n)=\Big(1+\frac 12 d_n-a_n^2-\frac 1{n}-a_nx-\Big(b_n-\frac 1{n}\Big)x^2\Big)
p_w(e_nx)+\frac {c(\mu)\theta}{n^{6/5}\sqrt{4-(e_nx)^2}}+\rho_n(x)
\end{equation}
for $x\in I_n-a_n$, where $\rho_n(x)$ is a continuous function such that
$0\le\rho_n(x)\le c(\mu)$ and $\int_{I_n-a_n}\rho_n(x)\,dx\le c(\mu)n^{-4}$.
Moreover,
\begin{equation}\label{asden2}
\int_{\mathbb R\setminus I_n}p_n(x)\,dx\le \frac {c(\mu)}{n^{6/5}}. 
\end{equation}
\end{theorem}
Here and in the sequel we denote by $\theta$ a real-valued quantity such that $|\theta|\le 1$.

\begin{corollary}\label{corth7.1}
Let $m_8(\mu)<\infty$ and $m_1(\mu)=0,\,m_2(\mu)=1$. 
If $m_3(\mu)\ne 0$, then
\begin{equation}\label{2.9}
\int_{\mathbb R}|p_n(x)-p_w(x)|\,dx=\frac {2|m_3(\mu)|}{\pi\sqrt n}+\theta c(\mu)(|a_n|^{3/2}+n^{-1}),
\quad n\ge n_1. 
\end{equation}
If $m_3(\mu)=0$, then
\begin{equation}\label{2.10}
\int_{\mathbb R}|p_n(x)-p_w(x)|\,dx=\frac {2|m_4(\mu)-2|}{\pi n}+\theta\frac{c(\mu)}{n^{6/5}},
\quad n\ge n_1. 
\end{equation}
\end{corollary}

In \cite{ChG:2011} we proved analogue results
for bounded free random variables.

A much stronger statement than the classical CLT -- the entropic central limit theorem --
indicates that, if for some $n_0$, or equivalently, for all $n\ge n_0$, $Y_n$ from (\ref{2.1})
have absolutely continuous distributions with finite entropies $h(Y_n)$, then there is convergence
of the entropies, $h(Y_n)\to h(Y)$, as $n\to \infty$, where $Y$ is a standard Gaussian random variable. 
This theorem is due to Barron~\cite{Ba:1986}. Recently Bobkov, Chistyakov and G\"otze~\cite{BChG:2011}
found the rate of convergence in the classical entropic CLT.

Recall that, if the random variable $X$ has density $f$, then the classical entropy of a
distribution of $X$ is defined as $h(X)=-\int_{\mathbb R}f(x)\log f(x)\,dx$,
provided the positive part of the integral is finite. Thus we have $h(X)\in[-\infty,\infty)$.

Let $\nu$ be a probability measure on $\mathbb R$. The quantity
\begin{equation}\notag
\chi(\nu)=\int\int_{\mathbb R\times \mathbb R}\log|x-y|\,\nu(dx)\nu(dy)+\frac 34+\frac 12\log 2\pi,
\end{equation}
called free entropy, was discovered by Voiculescu in~\cite{Vo:1993}. Free entropy $\chi$ behaves like 
the classical entropy $h$. 
In particular, the free entropy is maximized 
by the standard semicircular law $w$ with the value $\chi(w)=\frac 12\log 2\pi e$ among all 
probability measures with variance one~\cite{HiPe:2000}, \cite{Vo:1997}. Wang~\cite{Wa:2010}
has proved a~free analogue of Barron's result. We give the rate of convergence in the free CLT
for free random variables with a~finite moment of order 8. Previously we proved in \cite{ChG:2011} 
an analogous result for bounded free random variables.
\begin{corollary}\label{corth7.2}
Let $m_8(\mu)<\infty$ and $m_1(\mu)=0,\,m_2(\mu)=1$. Then, for $n\ge n_1$,
\begin{equation}\notag
\chi(\mu_n)=\int\int_{\mathbb R\times \mathbb R}\log|x-y|\,p_n(x)p_n(y)\,dxdy+\frac 34+\frac 12\log 2\pi
=\chi(w)-\frac {m_3^2(\mu)}{6n}+\theta\frac{c(\mu)}{n^{6/5}}.  
\end{equation} 
\end{corollary}

Suppose that the measure $\nu$ has a density $p$ in $L^3(\mathbb R)$. Then, following Voiculescu~\cite{Vo:1997},
the free Fisher information is
\begin{equation}\notag
\Phi(\nu)=\frac{4\pi^2}3\int_{\mathbb R}p(x)^3\,dx.
\end{equation}
It is well-known that $\Phi(w)=1$.
We obtain the following result for free random variables with finite moment of order 8.
\begin{corollary}\label{corth7.3}
Let $m_8(\mu)<\infty$ and $m_1(\mu)=0,\,m_2(\mu)=1$. Then
\begin{equation}\label{2.11}
\Phi(\mu_n)=\int_{\mathbb R}p_n(x)^3\,dx=\Phi(w)+\frac{m_3^2(\mu)}n+\theta\frac{c(\mu)}{n^{6/5}},\quad n\ge n_1.  
\end{equation} 
\end{corollary}

\section{Auxiliary results}

We need results about some classes of analytic functions
(see {\cite{Akh:1965}, Section~3.

The~class $\mathcal N$ (Nevanlinna, R.) is the~class of analytic 
functions $f(z):\mathbb C^+\to\{z: \,\Im z\ge 0\}$.
For such functions there is an~integral representation
\begin{equation}\label{3.1}
f(z)=a+bz+\int\limits_{\mathbb R}\frac{1+uz}{u-z}\,\tau(du)=
a+bz+\int\limits_{\mathbb R}\Big(\frac 1{u-z}-\frac u{1+u^2}\Big)(1+u^2)
\,\tau(du),\quad z\in\mathbb C^+,
\end{equation}
where $b\ge 0$, $a\in\mathbb R$, and $\tau$ is a~non-negative finite
measure. Moreover, $a=\Re f(i)$ and $\tau(\mathbb R)=\Im f(i)-b$.   
From this formula it follows that 
$f(z)=(b+o(1))z$
for $z\in\mathbb C^+$
such that $|\Re z|/\Im z$ stays bounded as $|z|$ tends to infinity (in other words
$z\to\infty$ non tangentially to $\mathbb R$).
Hence if $b\ne 0$, then $f$ has a~right inverse $f^{(-1)}$ defined
on the~region 
$
\Gamma_{\alpha,\beta}:=\{z\in\mathbb C^+:|\Re z|<\alpha \Im z,\,\Im z>\beta\}
$
for any $\alpha>0$ and some positive $\beta=\beta(f,\alpha)$.

A~function $f\in\mathcal N$ admits the~representation
\begin{equation}\label{3.3}
f(z)=\int\limits_{\mathbb R}\frac{\sigma(du)}{u-z},\quad z\in\mathbb C^+,
\end{equation}
where $\sigma$ is a~finite non-negative measure, if and only if
$\sup_{y\ge 1}|yf(iy)|<\infty$. Moreover $\sigma(\mathbb R)=-\lim_{y\to+\infty}iyf(iy)$.

For $\mu\in\mathcal M$, consider its Cauchy transform $G_{\mu}(z)$
\begin{equation}\label{3.5a}
G_{\mu}(z)=\int_{\mathbb R}\frac{\mu(du)}{z-u},\quad z\in\mathbb C^+. 
\end{equation}

The measure $\mu$ can be recovered from $G_{\mu}(z)$ as
the weak limit of the measures
\begin{equation}\notag
\mu_y(dx)=-\frac 1{\pi}\Im G_{\mu}(x+iy)\,dx,\quad x\in\mathbb R,\,\,y>0,
\end{equation}
as $y\downarrow 0$. If the function $\Im G_{\mu}(z)$ is continuous at $x\in\mathbb R$,
then the probability distribution function $D_{\mu}(t)=\mu((-\infty,t))$ is differentiable
at $x$ and its derivative is given by 
\begin{equation}\label{3.4}
D_{\mu}'(x)=-\Im G_{\mu}(x)/\pi. 
\end{equation}
This inversion formula
allows to extract the density function of the measure $\mu$ from its Cauchy transform.

Following Maassen~\cite{Ma:1992} and Bercovici and 
Voiculescu~\cite{BeVo:1993}, we shall consider in the~following
the~ {\it reciprocal Cauchy transform}
\begin{equation}\label{3.5}
F_{\mu}(z)=\frac 1{G_{\mu}(z)}.
\end{equation}
The~corresponding class of reciprocal Cauchy
transforms of all $\mu\in\mathcal M$ will be denoted by $\mathcal F$.
This class coincides with the~subclass of Nevanlinna functions $f$
for which $f(z)/z\to 1$ as $z\to\infty$ non tangentially to $\mathbb R$.

The~following lemma is well-known, see~\cite{Akh:1965}, Th. 3.2.1, p. 95. 
\begin{lemma}\label{3.4abl}
Let $\mu$ be a~probability measure such that
\begin{equation}\label{3.4abl1}
m_k=m_k(\mu):=\int\limits_{\Bbb R}u^k\,\mu(du)<\infty,\qquad k=0,1,\dots,2n,\quad n\ge 1.
\end{equation}
Then the~following relation holds
\begin{equation}\label{3.4abl2}
\lim_{z\to\infty}z^{2n+1}\Big(G_{\mu}(z)-\frac 1z-\frac{m_1}{z^2}-
\dots-\frac{m_{2n-1}}{z^{2n}}\Big)=m_{2n}
\end{equation}
uniformly in the~angle $\delta\le\arg z\le\pi-\delta$, 
where $0<\delta<\pi/2$.

Conversely, if for some function $G(z)\in\mathcal N$ the~relation $(\ref{3.4abl2})$
holds with real numbers $m_k$ for $z=iy,y\to\infty$, then $G(z)$ admits
the~representation~$(\ref{3.5a})$, where $\mu$ is a probability measure with moments $(\ref{3.4abl1})$.
\end{lemma}

As shown before, $F_{\mu}(z)$ admits the representation (\ref{3.1}) with $b=1$.
From Lemma~\ref{3.4abl} the following proposition is immediate.
\begin{proposition}\label{3.3^*pro} 
In order that a~probability measure $\mu$ satisfies the assumption $(\ref{3.4abl1})$, where $m_1(\mu)=0$, 
it is necessary and sufficient that
\begin{equation}\label{3.3^*proa}
F_{\mu}(z)=z+\int_{\mathbb R}\frac{\tau(du)}{u-z},\quad z\in\mathbb C^+, 
\end{equation}
where $\tau$ is a nonnegative measure such that $m_{2n-2}(\tau)<\infty$. Moreover
\begin{equation}\label{3.3^*prob}
m_k(\mu)=\sum_{l=1}^{[k/2]}\sum_{s_1+\dots+s_l=k-2,\,s_j\ge 0}m_{s_1}(\tau)\dots m_{s_l}(\tau),\quad k=2,\dots,2n. 
\end{equation}
\end{proposition}

Voiculescu~\cite{Vo:1993} showed for compactly supported probability measures that
there exist unique functions $Z_1, Z_2\in\mathcal F$ such that
$G_{\mu_1\boxplus\mu_2}(z)=G_{\mu_1}(Z_1(z))=G_{\mu_2}(Z_2(z))$ 
for all $z\in\mathbb C^+$.
Using Speicher's combinatorial approach~\cite{Sp:1998} to freeness,
Biane~\cite{Bi:1998} proved this result in the~general case.

Chistyakov and G\"otze \cite {ChG:2005}, Bercovici and Belinschi~\cite{BelBer:2007},
Belinschi~\cite{Bel:2008}, 
proved, using complex analytic methods, that
there exist unique functions $Z_1(z)$ and $Z_2(z)$ in the~class 
$\mathcal F$ such that, for $z\in\mathbb C^+$, 
\begin{equation}\label{3.9}
z=Z_1(z)+Z_2(z)-F_{\mu_1}(Z_1(z))\quad\text{and}\quad
F_{\mu_1}(Z_1(z))=F_{\mu_2}(Z_2(z)). 
\end{equation}
The~function $F_{\mu_1}(Z_1(z))$ belongs again to the~class $\mathcal F$ and 
there exists 
$\mu\in\mathcal M$ such that
$F_{\mu_1}(Z_1(z)) =F_{\mu}(z)$, where $F_{\mu}(z)=1/G_{\mu}(z)$ and 
$G_{\mu}(z)$ is the~Cauchy transform as in (\ref{3.5a}).  
The~measure $\mu$ depends on $\mu_1$ and $\mu_2$ only and $\mu=\mu_1\boxplus\mu_2$.

Specializing to $\mu_1=\mu_2=\dots=\mu_n=\mu$ write $\mu_1\boxplus\dots\boxplus\mu_n=
\mu^{n\boxplus}$.
The~relation (\ref{3.9}) admits the~following
consequence (see for example \cite{ChG:2005}, Section 2, Corollary 2.3).

\begin{proposition}\label{3.3pro}
Let $\mu\in\mathcal M$. There exists a~unique function $Z\in\mathcal F$ 
such that
\begin{equation}\label{3.10}
z=nZ(z)-(n-1)F_{\mu}(Z(z)),\quad z\in\mathbb C^+,
\end{equation}
and $F_{\mu^{n\boxplus}}(z)=F_{\mu}(Z(z))$.
\end{proposition}

Using the~representation (\ref{3.1}) for $F_{\mu}(z)$ we obtain
\begin{equation}\label{7.8}
F_{\mu}(z)=z+\Re F_{\mu}(i)+\int\limits_{\mathbb R}\frac {(1+uz)\,\tau(du)}{u-z},
\quad z\in\mathbb C^+,
\end{equation}
where $\tau$ is a~nonnegative measure such that $\tau(\mathbb R)=\Im F_{\mu}(i)-1$.
Denote $z=x+iy$, where $x,y\in\mathbb R$. We see that, for $\Im z>0$, 
\begin{equation}\notag
\Im \Big(nz-(n-1)F_{\mu}(z)\Big)=y\Big(1-(n-1)I_{\mu}(x,y)\Big),\quad\text{where}\quad                                                                                                                                                                                                                                                                                                                                                                                                                                                                                                                                                                                                                             
I_{\mu}(x,y):=\int\limits_{\mathbb R}\frac{(1+u^2)\,\tau(du)}{(u-x)^2+y^2}.
\end{equation}
For every real fixed $x$, consider the~equation
\begin{equation}\label{7.9}
 y\Big(1-(n-1)I_{\mu}(x,y)\Big)=0,\quad y>0.
\end{equation}
Since $y\mapsto I_{\mu}(x,y),\,y>0$, is positive and monotone, and decreases to $0$ as $y\to\infty$,
it is clear that the~equation (\ref{7.9}) has at most one positive solution. 
If such a~solution exists, denote it
by $y_n(x)$.
Note that (\ref{7.9}) does not have a~solution $y>0$ for any given $x\in\mathbb R$ if and only if
$I_{\mu}(x,0)\le 1/(n-1)$.
Consider the~set $S:=\{x\in\mathbb R:I_{\mu}(x,0)\le 1/(n-1)\}$. We put $y_n(x)=0$ for $x\in S$. 
We proved in~\cite{ChG:2011}, Section~3, p.13, that
the~curve $\gamma_n$ given by the~equation $z=x+iy_n(x),\,x\in\mathbb R$, is continuous and simple.

Consider the~open domain $\tilde{D}_n:=\{z=x+iy,\,x,y\in\mathbb R: y>y_n(x)\}$.
\begin{lemma}\label{l7.4}                                                                                                                
Let $Z\in\mathcal F$ be the~solution of the~equation $(\ref{3.10})$. The function $Z(z)$ maps $\mathbb C^+$
conformally onto $\tilde{D}_n$.  
Moreover the~function $Z(z),\,z\in\mathbb C^+$, 
is continuous up to the~real axis and it establishes a homeomorphism between the real axis and 
the~curve $\gamma_n$.                                                                                                                                      
\end{lemma}
This lemma was proved in~\cite{ChG:2011} (see Lemma~3.4). The following lemma was proved as well 
in~\cite{ChG:2011} (see Lemma~3.5).

\begin{lemma}\label{l7.5}
Let $\mu$ be a~probability measure such that $m_1(\mu)=0,m_2(\mu)=1$. Assume that 
$\int_{|u|>\sqrt{(n-1)/8}}u^2\,\mu(du)\le 1/10$ for some
positive integer $n\ge 10^3$.
Then the~following inequality holds
\begin{equation}\label{7.11*}
|Z(z)|\ge \sqrt{(n-1)/8},\qquad z\in\mathbb C^+, 
\end{equation} 
where $Z\in\mathcal F$ is the~solution of the~equation $(\ref{3.10})$.
\end{lemma}

\section{Free Meixner measures}

Consider the~three-parameter family of probability measures $\{\mu_{a,b,d}:
a\in\mathbb R, b<1, d<1\}$ with the~reciprocal Cauchy transform
\begin{equation}\label{2.3h}
\frac 1{G_{\mu_{a,b,d}}(z)}=a+\frac 12\Big((1+b)(z-a)+\sqrt{(1-b)^2(z-a)^2-4(1-d)}\Big),
\quad z\in\mathbb C,
\end{equation}
which we will call the~free centered (i.e. with mean zero) Meixner measures. 
In this formula we choose the~branch of the~square root determined by the~condition
$\Im z>0$ implies $\Im (1/G_{\mu_{a,b,d}}(z))\ge 0$.
These measures are counterparts of the~classical measures discovered by Meixner~\cite{Me:1934}.
The~free Meixner type measures occurred in many places in the~literature, see for example 
\cite{BoBr:2006}, \cite{SaYo:2001}.

Saitoh and Yoshida~\cite{SaYo:2001} have proved that
the~absolutely continuous part of the~free Meixner measure $\mu_{a,b,d},a\in\mathbb R,b<1,d<1$, is 
given by
\begin{equation}\label{2.3g}
\frac{\sqrt{4(1-d)-(1-b)^2(x-a)^2}}{2\pi f(x)},
\end{equation}
when $a-2\sqrt{1-d}/(1-b)\le x\le a+2\sqrt{1-d}/(1-b)$, where 
\begin{equation}\notag
f(x):=bx^2+a(1-b)x+1-d;  
\end{equation}

Saitoh and Yoshida proved as well that for $0\le b<1$ the (centered) 
free Meixner measure
$\mu_{a,b,d}$ is $\boxplus$-infinitely divisible.

As we have shown in~\cite{ChG:2011}, Section 4, it follows from Saitoh and Yoshida's results that the~probability measure
$\mu_{a_n,b_n,d_n}$ with the parameters $a_n,b_n,d_n$ from (\ref{2.3****}) is $\boxplus$-infinitely divisible 
and it is absolutely continuous with a~density of the~form (\ref{2.3g}) where
$a=a_n, b=b_n,d=d_n$ for sufficiently large $n\ge n_1(\mu)$.

\section{Passage to measures with bounded supports}

Let us assume that $\mu\in\mathcal M$ and $m_8(\mu)<\infty$. In addition let $m_1(\mu)=0$ and $m_2(\mu)=1$.
By Proposition~\ref{3.3pro}, there exists $Z(z)\in\mathcal F$
such that (\ref{3.10}) holds, and $F_{\mu^{n\boxplus}}(z)=F_{\mu}(Z(z))$.
Hence $F_{\mu_n}(z)=F_{\mu}(\sqrt n S_n(z))/\sqrt n,\,z\in\mathbb C^+$, 
where $S_n(z):=Z(\sqrt n z)/\sqrt n$. Since $m_1(\mu)=0,\,m_2(\mu)=1$ and $m_8(\mu)<\infty$, 
by Proposition~\ref{3.3^*pro},
we have the representation
\begin{equation}\notag
F_{\mu}(z)=z+\int_{\mathbb R}\frac{\tau(du)}{u-z},\quad z\in\mathbb R, 
\end{equation}
where $\tau$ is a nonnegative measure such that $\tau(\mathbb R)=1$ and $m_6(\tau)<\infty$. 
Consider a function
$$
F(z)=z+\int_{|u|\le \sqrt{n-1}/\pi}\frac{\tau(du)}{u-z},\quad z\in\mathbb R.
$$ 
This function belongs to the class $\Cal F$ and therefore there exists the probability measure $\mu^*$
such that $F_{\mu^*}(z)=F(z),\,z\in\mathbb R$. The probability measure $\mu^*$ depends of cause on
$n$. Moreover we conclude from the inversion formula that $\mu^*([-\sqrt{n-1}/3,\sqrt{n-1}/3])=1$
for $n\ge n_1(\mu)$.
By Proposition~\ref{3.3^*pro}, we see as well that 
\begin{align}\label{Pas5.1}
&m_1(\mu^*)=0,\quad m_2(\mu)-m_2(\mu^*)=
\tau(\mathbb R\setminus[-\sqrt{n-1}/\pi,\sqrt{n-1}/\pi])\le c(\mu)n^{-3},\notag\\
&\text{and}\quad |m_j(\mu)-m_j(\mu^*)|\le c(\mu)n^{-3+(j-2)/2},\quad j=3,\dots,8.
\end{align}

Let $X^*,X_1^*,X_2^*,\dots$ be free identically distributed random variables such that
$\Cal L(X^*)=\mu^*$. 
Denote $\mu_n^*:=\Cal L((X_1^*+\dots+X_n^*)/\sqrt {n})$. As before, 
by Proposition~\ref{3.3pro}, there exists $W(z)\in\mathcal F$
such that (\ref{3.10}) holds with $Z=W$ and $\mu=\mu^*$, and $F_{(\mu^*)^{n\boxplus}}(z)=F_{\mu^*}(W(z))$.
Hence $F_{\mu_n^*}(z)=F_{\mu^*}(\sqrt {n} T_n(z))/\sqrt {n},\,z\in\mathbb C^+$, 
where $T_n(z):=W(\sqrt {n} z)/\sqrt {n}$. In the sequel we need information about
the behaviour of the functions $T_n(z)$ and $S_n(z)$. By Lemma~\ref{l7.4}, 
these functions are continuous up to the real axis for $n\ge n_1(\mu)$. 
Their values for $z=x\in\mathbb R$ we denote by 
$T_n(x)$ and $S_n(x)$, respectively.
In order to formulate next results about $T_n(z)$ 
we introduce some notations. Denote by $M_n(z)$ the reciprocal Cauchy transform of the free Meixner
measure $\mu_{a_n,b_n,d_n}$ with the parameters $a_n,b_n$ and $d_n$ from (\ref{2.3****}), i.e.,
$$
M_{n}(z):=a_n+\frac 12\Big(\big(1+b_n\big)(z-a_n)+
\sqrt{\big(1-b_n\big)^2(z-a_n)^2-4\big(1-d_n\big)}\Big),\quad z\in\mathbb C^+.
$$
Denote by $D_n$ the rectangle
\begin{equation}\notag
\Big\{z\in\mathbb C:0<\Im z\le 3,|\Re z-a_n|\le \frac 2{e_n}-h\Big\}.
\end{equation}
where $h:=c_2(\mu)n^{-3/2}$ with some constant $c_2(\mu)$.

Repeating step by the step the arguments of Section~7 (see Subsections 7.2--7.7) of our paper~\cite{ChG:2011}
we establish the following result.
\begin{theorem}\label{th4a}
Let $\mu\in\mathcal M$ such that $m_8(\mu)<\infty$ and $m_1(\mu)=0,\,m_2(\mu)=1$.
Then
there exists the constant $c_2(\mu)$ such that the following relation holds, for $z\in D_n$ and
$n\ge n_1(\mu)$,
\begin{align}
T_n(z)&=M_{n}(z)+\frac{r_{n1}(z)}{n^{3/2}\sqrt{(e_n(z-a_n))^2-4}}, \label{th4.1}\\
G_{\mu_n^*}(z)&=\frac 1{T_n(z)}+\frac 1{nT_n(z)^3}+\frac{r_{n2}(z)}{n^{3/2}},\label{th4.2}
\end{align}
where $|r_{nj}(z)|\le c(\mu),\,j=1,2$.
\end{theorem}

In Section~7 of this paper we shall give more detailed explanations of the proof of this theorem.

By Lemmas~\ref{l7.4},~\ref{l7.5},
$|T_n(z)|\ge 1.03/3$ for $z\in \mathbb C^+\cup \mathbb R$ and for $n\ge n_1(\mu)$. It is obvious that
the same estimate holds for $M_n(z)$. Since 
\begin{equation}\label{th4.2^*}
G_{\mu_n^*}(z)=\sqrt {n} G_{\mu^*}(\sqrt {n} T_n(z))=\int\limits_{[-\sqrt{n-1}/3,\sqrt{n-1}/3]}
\frac{\mu^*(du)}{T_n(z)-u/\sqrt {n}},\qquad z\in\mathbb C^+, 
\end{equation}
we conclude that $G_{\mu_n^*}(z)$ is a continuous function up to the real axis. Denote its value for real $x$  
by $G_{\mu_n^*}(x)$.
Denote $G_{\hat{\mu}_n}(z):=1/T_n(z),\,z\in\mathbb C^+$. This function is continuous up to the real axis
as well. Therefore
$\hat{\mu}_n$ and $\mu_n^*$ are absolutely continuous measures with continuous densities
$\hat{p}_n(x)$ and $p_n^*(x)$, respectively,
\begin{align}
&\hat{p}_n(x)=-\lim_{\varepsilon\downarrow 0}\frac 1{\pi}\Im \frac 1{T_n(x+i\varepsilon)} 
=-\frac 1{\pi}\Im \frac 1{T_n(x)},\notag\\
&p_n^*(x)=-\lim_{\varepsilon\downarrow 0}\frac 1{\pi}\Im G_{\mu_n^*}(x+i\varepsilon)
=-\frac 1{\pi}\Im G_{\mu_n^*}(x).\notag
\end{align}

In addition, $\hat{p}_n(x)\le 1$ and $p_n^*(x)\le 50$ for all $x\in \mathbb R$ and $n\ge n_1(\mu)$.  
Let $c_2(\mu)$ be the constant defined in Theorem~\ref{th4a} and as before $h=c_2(\mu)n^{-3/2}$.
\begin{theorem}\label{th5}
Let $\mu\in\mathcal M$ such that $m_8(\mu)<\infty$ and $m_1(\mu)=0,\,m_2(\mu)=1$.
Then, for $x\in \hat{I}_n:=\{x\in\mathbb R:|x-a_n|\le\frac 2{e_n}-h\}$ and $n\ge n_1(\mu)$, 
the following relation holds
\begin{align}\label{loc.6}  
p_n^*(x)=v_n(x-a_n)
+\frac {c(\mu)\theta}{n^{3/2}\sqrt{4-(e_n(x-a_n))^2}}, 
\end{align}
where
\begin{align}\notag
v_n(x):=\Big(1+\frac{d_n}2-a_n^2-\frac 1{n}-a_nx-\Big(b_n-a_n^2-\frac 1{n}\Big)x^2\Big)
p_w(e_nx),\quad x\in\mathbb R. 
\end{align} 
\end{theorem}

Before to prove Theorem~\ref{th5} we shall make the remark.
\begin{remark}
Using the arguments of this paper and the paper\cite{ChG:2011} one can prove the following result.

Let $\mu\in\mathcal M$ such that $m_4(\mu)<\infty$ and $m_1(\mu)=0,\,m_2(\mu)=1$.
Then, for $x\in \tilde{I}_n:=\{x\in\mathbb R:|x-a_n|\le\frac 2{e_n}-o(\frac 1n)\}$, we have
\begin{equation}\notag
p_n^*(x)=v_n(x-a_n)+\frac {o(1/n)}{\sqrt{4-(e_n(x-a_n))^2}}. 
\end{equation}  
\end{remark}

\begin{proof}
We shall use the following estimate, for $x\in\mathbb R$,   
\begin{align}\label{loc.1}
|p_n^*(x)-p_{\mu_{a_n,b_n,d_n}}(x)-\frac 1n q_{n}(x)|&\le |\hat{p}_n(x)-p_{\mu_{a_n,b_n,d_n}}x)|
+|p_n^*(x)-\hat{p}_n(x)-\frac 1n \hat{q}_{n}(x)|\notag\\
&+\frac 1n|q_{n}(x)-\hat{q}_n(x)|, 
\end{align}
where 
\begin{equation}\notag
q_{n}(x):=-\frac 1{\pi}\Im \frac 1{M_n(x)^3},\quad\text{and}\quad
\hat{q}_{n}(x):=-\frac 1{\pi}\Im \frac 1{T_n(x)^3}.
\end{equation}
By (\ref{th4.1}), we easily obtain the upper bound, for $x\in \hat{I}_n$
and $n\ge n_1(\mu)$,
\begin{equation}\label{loc.2}
|\hat{p}_n(x)-p_{\mu_{a_n,b_n,d_n}}x)|\le \frac {c(\mu)}{n^{3/2}\sqrt{4-(e_n(x-a_n))^2}} 
\end{equation}
and, by (\ref{th4.2}), we have
\begin{equation}\label{loc.3}
|p_n^*(x)-\hat{p}_n(x)-\frac 1n \hat{q}_{n}(x)|\le 
\frac {c(\mu)}{n^{3/2}}. 
\end{equation}

Since, by (\ref{2.3g}),
\begin{equation}
p_{\mu_{a_n,b_n,d_n}}(x):=\frac {\sqrt{4(1-d_n)-(1-b_n)^2(x-a_n)^2}}{2\pi(b_nx^2+a_n(1-b_n)x+1-d_n)},
\quad x\in \tilde{I}_n:=[a_n-2/e_n,a_n+2/e_n], 
\end{equation}
we easily conclude that
\begin{equation}\label{loc.4}
p_{\mu_{a_n,b_n,d_n}}(x)=\Big(1+\frac{d_n}2-a_n^2-a_n(x-a_n)-(b_n-a_n^2)(x-a_n)^2\Big)
p_w(e_n(x-a_n))+c(\mu)\theta n^{-3/2} 
\end{equation}
for $x\in \tilde{I}_n$. Using (\ref{th4.1}) and the lower bounds $|T_n(x)|\ge 1.03/3,\, |M_n(x)|\ge 1.03/3$,
for $x\in\mathbb R$, we obtain
\begin{align}\label{loc.4a}
|q_n(x)-\hat{q}_n(x)|&\le \frac 1{\pi}|T_n(x)-M_n(x)|\Big(\frac 1{|M_n(x)T_n(x)^3|}+
\frac 1{|M_n(x)^2T_n(x)^2|}+\frac 1{|M_n(x)^3T_n(x)|}\Big)
\notag\\
&\le \frac { c(\mu)}{n^{3/2}\sqrt{4-(e_n(x-a_n))^2}},\quad x\in \hat{I}_n.   
\end{align}
On the other hand it is not difficult to show that
\begin{align}
q_{n}(x)&:=\frac 1{8\pi}
\sqrt{(4(1-d_{n})-(1-b_n)^2(x-a_n)^2)_+}\notag\\
&\times\frac{3((1+b_n)x+(1-b_n)a_n)^2+(1-b_n)^2(x-a_n)^2-
4(1-d_n)}{(b_nx^2+(1-b_n)a_n x+1-d_{n})^3},\quad x\in\mathbb R,\notag
\end{align}
which leads to the relation
\begin{equation}\label{loc.5}
q_{n}(x)=((x-a_n)^2-1)p_w(e_n(x-a_n))
+c(\mu)\theta (|a_n|+n^{-1})  
\end{equation}
for $x\in \tilde{I}_n$.

Applying (\ref{loc.2}), (\ref{loc.3}), (\ref{loc.4a}) and (\ref{loc.4}), (\ref{loc.5}) to (\ref{loc.1}) 
we arrive at the statement of the theorem.
\end{proof}

\section{Local asymptotic expansion}

First we note that Theorem~{\ref{th7}} follows immediately from Theorem~\ref{th5} and
the following auxiliary result.
\begin{theorem}\label{th4c}
Let $\mu\in\mathcal M$ such that $m_8(\mu)<\infty$ and $m_1(\mu)=0,\,m_2(\mu)=1$.
Then there exists a constant $c_1(\mu)$ in $(\ref{asden1})$ such that the following relation holds
\begin{equation}\notag
p_{n}(x)=p_{n}^*(x)+\frac{c(\mu)\theta}{n^{6/5}\sqrt{4-(e_n(x-a_n))^2}}+\rho_n(x),\quad x\in I_n,\,\,
n\ge n_1(\mu),
\end{equation} 
where $\rho_n(x)$ is a continuous function such that 
$0\le \rho_n(x)\le c(\mu)$ and $\int_{I_n}\rho_n(x)\,dx\le c(\mu)n^{-4}$.
\end{theorem}

We prove Theorem~\ref{th4c}, using the following auxiliary result on the closeness 
of the functions $S_n(z)$ and $T_n(z)$.
\begin{theorem}\label{th4b}
Let $\mu$ satisfy the assumptions of Theorem~${\ref{th4c}}$. Then
there exists a constant $c_1(\mu)$ in $(\ref{asden1})$, such that 
for $x\in I_n$ and $n\ge n_1(\mu)$,
\begin{equation}\label{th4.3}
|S_n(x)-T_n(x)|\le \frac{c(\mu)}{n^{6/5}\sqrt{4-(e_n(x-a_n))^2}}. 
\end{equation}
\end{theorem}
We shall prove this theorem in the next section. Return to the proof of Theorem~\ref{th4c}.

\begin{proof}
Represent the density $p_{n}(x)$ of the measure $\mu_n$ in the form
\begin{equation}\label{th4.4}
p_{n}(x)=p_{n1}(x)+p_{n2}(x),\quad x\in\mathbb R, 
\end{equation}
where $p_{nj}(x)\ge 0,\,x\in\mathbb R,\,j=1,2$, and, for $z\in\mathbb C^+$,
\begin{align}
I_1(z)&:=\int_{|u|\le \sqrt{n-1}/3}\frac{\mu(du)}{S_n(z)-u/\sqrt n}=\int_{\mathbb R}\frac{p_{n1}(u)\,du}{z-u},\notag\\
I_2(z)&:=\int_{|u|>\sqrt{n-1}/3}\frac{\mu(du)}{S_n(z)-u/\sqrt n}=\int_{\mathbb R}\frac{p_{n2}(u)\,du}{z-u}.\notag
\end{align} 
Since $\lim_{y\to+\infty}iyI_2(iy)=\int_{|u|>\sqrt{n-1}/3}\mu(du)=\int_{\mathbb R}p_{n2}(u)\,du$, we note that
\begin{equation}\label{th4.5}
\int_{\mathbb R}p_{n2}(u)\,du\le c(\mu)n^{-4}. 
\end{equation}
Since $S_n(x),\,x\in\mathbb R$, is a continuous function and $|S_n(x)|\ge 1.03/3$
for all $x\in\mathbb R$, 
we easily see that, for $x\in\mathbb R$,
\begin{align}\notag
p_{n1}(x)=-\Im \int_{|u|\le \sqrt{n-1}/3}\frac{\mu(du)}{S_n(x)-u/\sqrt n} 
\end{align}
and that $p_{n1}(x)$ is a continuous function on the real line. In view of (\ref{asden3}), 
$p_{n}(x)$ is a continuous function on the real line and
$p_n(x)\le c(\mu),\,x\in\mathbb R$, 
for $n\ge n_1(\mu)$. Therefore we conclude from (\ref{th4.4}) that $p_{n2}(x)$ is a continuous function 
on the real line and $p_{n2}(x)\le c(\mu),\,x\in\mathbb R$, for the same $n$.

Now we may write
\begin{align}
p_{n}^*(x)-p_{n1}(x)&=I_{3,1}(x)+I_{3,2}(x)\notag\\
&:=\Im\Big(\int_{|u|\le \sqrt{n-1}/3}\frac{\mu(du)}{S_n(x)-u/\sqrt n}-
\int_{|u|\le \sqrt{n-1}/3}\frac{\mu(du)}{T_n(x)-u/\sqrt n}\Big)\notag\\
&+\Im\int_{|u|\le \sqrt{n-1}/3}\frac{(\mu-\mu^*)(du)}{T_n(x)-u/\sqrt n}. \label{th4.6^*}
\end{align}
By Theorem~\ref{th4b} and the inequalities $|S_n(x)-u/\sqrt n|\ge 0.01$, $|T_n(x)-u/\sqrt n|\ge 0.01$
for $x\in\mathbb R$ and $|u|\le\sqrt{n-1}/3$, we conclude that, for $x\in I_n$, 
\begin{equation}\label{th4.6}
|I_{3,1}(x)|\le \int_{|u|\le\sqrt{n-1}/3}\frac{|S_n(x)-T_n(x)|\,\mu(du)}{|S_n(x)-u/\sqrt n||T_n(x)-u/\sqrt n|}
\le \frac{c(\mu)}{n^{6/5}\sqrt{4-(e_n(x-a_n))^2}}.
\end{equation}
On the other hand we note that, for $x\in\mathbb R$,
\begin{equation}\notag
I_{3,2}(x)=\Im\Big(\frac{m_2(\mu)-m_2(\mu^*)}{T_n^3(x)n}+\frac{m_3(\mu)-m_3(\mu^*)}{T_n^4(x)n^{3/2}}
+\frac 1{T_n^4(x)n^{2}}\int\limits_{|u|\le \sqrt{n-1}/3}\frac{u^4(\mu-\mu^*)(du)}{T_n(x)-u/\sqrt n}\Big). 
\end{equation}
Using (\ref{Pas5.1}), we obtain
\begin{equation}\label{th4.7^*}
|I_{3,2}(x)|\le \frac{c(\mu)}{n^{2}},\quad x\in\mathbb R. 
\end{equation}
Applying (\ref{th4.6}) and (\ref{th4.7^*}) to (\ref{th4.6^*}), we have, for $x\in I_n$, 
\begin{equation}\label{th4.8}
|p_{n}^*(x)-p_{n1}(x)|\le \frac{c(\mu)}{n^{6/5}\sqrt{4-(e_n(x-a_n))^2}}. 
\end{equation}
It remains to note that the statement of the theorem immediately follows from (\ref{th4.4}), (\ref{th4.5})
and (\ref{th4.8}).

\end{proof}

\section{Proof of Theorem~\ref{th4a}}

In this section we show how using arguments of Section 7 from \cite{ChG:2011}
one can obtain a~proof of Theorem~\ref{th4a}. 

Repeating the arguments of Subsection~7.2 we deduce that $T_n(z)$ satisfies the functional equation,
for $z\in\mathbb C^+$,
\begin{equation}\label{5.5}
T_n^5(z)-zT_n^4(z)+m_2(\mu^*)T_n^3(z)
+\frac{\zeta_{n2}(z)}{\sqrt n}T_n^2(z)
+\frac{\zeta_{n3}(z)}n T_n(z)-\frac{\zeta_{n4}(z)z}{n^2} =0,
\end{equation}
where 
$
\zeta_{n1}(z):=\int_{\mathbb R}\frac{u^5\,\mu^*(du)}{W(\sqrt nz)-u},
$
$\zeta_{n2}(z):=m_3(\mu^*)-z/\sqrt n$, $\zeta_{n3}(z)(z):=m_4(\mu^*)+\zeta_{n1}(z)-zm_3(\mu^*)/\sqrt n$ and
$\zeta_{n4}(z)(z):=m_4(\mu^*)+\zeta_{n1}(z)$.
As in Subsection~7.3 from \cite{ChG:2011} we obtain estimates for the functions $\zeta_{nj}(z),j=1,2,3,4$ in the domain
$D^*:=\{|\Re z|\le 4,\,0<\Im z\le 3\}$
\begin{equation}\label{5.5a}
|\zeta_{n1}(z)|\le c(\mu)n^{-1/2},\quad \sum_{j=2}^4 |\zeta_{nj}(z)|\le c(\mu).
\end{equation}

For every fixed $z\in\mathbb C^+$
consider the~equation
\begin{equation}\label{5.4c}
Q(z,w):=w^5-zw^4+m_2(\mu^*)w^3+\frac{\zeta_{n2}(z)}{\sqrt n} w^2
+\frac{\zeta_{n3}(z)}n w-\frac{\zeta_{n4}(z)z}{n^2} =0.
\end{equation}
Denote the~roots of the equation (\ref{5.4c}) by $w_j=w_j(z),\,j=1,\dots,5$.

As in Subsection~7.4 \cite{ChG:2011} we can show that for every fixed $z\in D^*$ 
the~equation $Q(z,w)=0$ 
has three roots, say $w_j=w_j(z),\,j=1,2,3$, such that
\begin{equation}\label{5.6}
|w_j|<r':=c_3(\mu) n^{-1/2},
\quad j=1,2,3,
\end{equation}
and two roots, say $w_j,\,j=4,5$, such that $|w_j|\ge r'$ for $j=4,5$.

Represent $Q(z,w)$ in the~form
$$
Q(z,w)=(w^2+s_1w+s_2)(w^3+g_1w^2+g_2w+g_3),
$$
where $w^3+g_1w^2+g_2w+g_3=(w-w_1)(w-w_2)(w-w_3)$.
From this formula we derive the~relations
\begin{align}\label{5.7}
s_1+g_1=-z,\quad &s_2+s_1g_1+g_2=m_2(\mu^*),\quad s_2g_1+s_1g_2+g_3=
\frac{\zeta_{n2}(z)}{\sqrt n},
\notag \\
&s_2g_2+s_1g_3=\frac {\zeta_{n3}(z)}n,\quad s_2g_3=-\frac{\zeta_{n4}(z)z}{n^2}.
\end{align}
By Vieta's formulae and (\ref{5.6}), note that
\begin{equation}\label{5.8a}
|g_1|\le 3r',\quad |g_2|\le 3(r')^2,\quad
|g_3|\le (r')^3.
\end{equation}
Now we obtain from (\ref{5.7}) and (\ref{5.8a}) the~following bounds,
for $z\in D_1$,
\begin{equation}\label{5.8}
|s_1|\le 5+3r',\quad |m_2(\mu^*)-s_2|\le 3r'(4r'+5)\le 16r'\le\frac 12.
\end{equation}
Then we conclude from 
(\ref {Pas5.1}), (\ref{5.5a}), (\ref{5.7})--(\ref{5.8}) that, for the~same $z$,
\begin{align}\label{5.9}
\Big|g_2-\frac{\zeta_{n4}(z)}{n}\Big| &\le \Big|g_2-\frac{\zeta_{n3}(z)}{n}\Big|+\frac{|m_3(\mu^*)||z|}{n^{3/2}}
\le \frac{|s_1|}{|s_2|}|g_3|+\frac{|s_2-m_2(\mu^*)|}{|s_2|}\frac{|\zeta_{n3}(z)|}n+(r')^3\notag\\
&\le 11(r')^3+8(r')^3+(r')^3=20(r')^3\le c(\mu)n^{-3/2}.
\end{align}
Now repeating the arguments of Subsection~7.4 we deduce the inequality
\begin{equation}\label{5.10}
 |g_1-a_n-b_n z|\le c(\mu)n^{-3/2}, \quad z\in D^*.
\end{equation}
To find the~roots $w_4$ and $w_5$, we need to solve the~equation
$w^2+s_1w+s_2=0$. Using (\ref{5.7}), we have, 
for $j=4,5$,
\begin{align}\label{5.11}
w_j&=\frac 12\Big(-s_1+(-1)^j\sqrt{s_1^2-4s_2}\Big)\notag\\&=\frac 12\Big(
z+g_1+(-1)^j\sqrt{(z+g_1)^2-4(m_2(\mu^*)+(z+g_1)g_1-g_2)}\Big)\notag\\
&=\frac 12\Big(z+g_1+(-1)^j\sqrt{(z-g_1)^2-4m_2(\mu^*)-4(g_1^2-g_2)}\Big)
=\frac 12 r_{n3}(z)+a_n+\notag\\
&+\frac 12\Big(\big(1+b_n\big)(z-a_n)
+(-1)^j\sqrt{\big(1-b_n\big)^2(z-a_n)^2-4\big(1-d_n\big)
+r_{n4}(z)}\Big),
\end{align}
where 
\begin{align}
r_{n3}(z)&:=g_1-a_n-b_n(z-a_n),\notag\\
r_{n4}(z)&:=
-3r_{n3}^2(z)-2r_{n3}(z)(4a_n+(1+3b_n)(z-a_n))+4(1-m_2(\mu^*))\notag\\
&+4(g_2-m_4(\mu)/n)-4b_n(z-a_n)(2a_n+b_n(z-a_n)).\notag 
\end{align}

The quantities $r_{n3}(z)$ and $r_{n4}(z)$ admit the bound (see Subsection~7.4 and 7.5 from \cite{ChG:2011})
\begin{equation}\label{5.12}
|r_{n3}(z)|+|r_{n4}(z)|\le c(\mu)n^{-3/2}, \quad z\in D^*.
\end{equation}
We choose the~branch of the~analytic square root according to the~condition $\Im w_4(i)\ge 0$.

As in Subsection~7.6 from \cite{ChG:2011} we prove that $w_4(z)=S_n(z)$ for $z\in D_n$, where $h=c_2(\mu)n^{-3/2}$
with some sufficiently large constant $c_2(\mu)$ which does not depend on the constant $c(\mu)$ in
(\ref{5.12}). Since the constant $c_2(\mu)$ is sufficiently large, we have, by (\ref{5.12}),
\begin{equation}\label{5.13}
|r_{n3}(z)|/|((1-b_n)^2(z-a_n)^2-4(1-d_n)|\le 10^{-2},\quad z\in D_n. 
\end{equation}

For $z\in D_n$, using formula (\ref{5.11}) with $j=4$ for $S_n(z)$, we write
\begin{align}\label{5.14}
&M_n(z)-S_n(z)=-\frac 12 r_{n3}(z)\notag\\
&-\frac 12\frac{r_{n4}(z)}{\sqrt{(1-b_n)^2(z-a_n)^2-4(1-d_n)}
+\sqrt{(1-b_n)^2(z-a_n)^2-4(1-d_n)+\zeta_{n6}(z)}}. 
\end{align}
Using (\ref{5.13}) and the power expansion for the function $(1+z)^{1/2},\, |z|<1$, we easily
rewrite (\ref{5.14}) in the form
\begin{align}\label{5.15}
M_n(z)-S_n(z)=\frac{r_{n5}(z)}{\sqrt{(1-b_n)^2(z-a_n)^2-4(1-d_n)}},
\quad z\in D_n, 
\end{align}
where $|r_{n5}(z)|\le c(\mu)n^{-3/2}$. The relation (\ref{th4.1}) immediately follows from (\ref{5.15}). 

Using (\ref{th4.2^*}), we conclude that
\begin{equation}\label{5.2}
G_{\mu_n^*}(z)=\frac 1{T_n(z)}+\frac 1{nT_n^3(z)}
+\frac 1{n^{3/2}T_n^3(z)}
\int_{[-\sqrt{n-1}/3,\sqrt{n-1}/3]}\frac{u^4\,\mu^*(du)}{T_n(z)-u/\sqrt n},\quad z\in\mathbb C^+.
\end{equation}
Since $|T_n(z)|\ge 1.03/3$ for $z\in \mathbb C^+$ and $n\ge n_1(\mu)$, we arrive at (\ref{th4.2}).

\section{Proof of Theorem~\ref{th4b}}

In order to prove this theorem we need the following auxiliary results.
\begin{proposition}\label{th4bpro1}
For $z_1,z_2\in T_n(\mathbb C^+)$,
\begin{equation}\notag
|T_n^{(-1)}(z_1)-T_n^{(-1)}(z_2)|\le c(\mu)|z_1-z_2|,\quad n\ge n_1(\mu).
\end{equation} 
\end{proposition}
\begin{proof}
Using the formula
\begin{equation}\label{th4bpro1.1}
 T_n^{(-1)}(z)=nz-\frac{n-1}{\sqrt n}F_{\mu^*}(z\sqrt n)=z-\frac{n-1}{\sqrt n}\int_{-\sqrt{n-1}/\pi}^{\sqrt{n-1}/\pi}
\frac{\tau(du)}{u-z\sqrt n},\quad z\in\mathbb C^+,
\end{equation}
we have the relation, for $z_1,z_2\in\mathbb C^+$,
\begin{equation}\notag
T_n^{(-1)}(z_1)-T_n^{(-1)}(z_2)=(z_1-z_2)\Big(1-(n-1)\int_{-\sqrt{n-1}/\pi}^{\sqrt{n-1}/\pi}
\frac{\tau(du)}{(u-z_1\sqrt n)(u-z_2\sqrt n)}\Big). 
\end{equation}
Since, by Lemma~3.5, $|u-z\sqrt n|\ge 10^{-2}\sqrt{n}$ for $|u|\le \sqrt{n-1}/\pi$ and $z\in T_n(\mathbb C^+)$, we 
immediately arrive at the assertion of the proposition.
\end{proof}

Recalling the definition of the function $Z(z)$, we see, by Lemma~\ref{l7.4}, that the function $S_n(z)$ 
maps $\mathbb C^+$ conformally onto $\hat{D}_n$, where 
$\hat{D}_n:=\{z=x+iy,x,y\in\mathbb R:y>y_n(x\sqrt n)/\sqrt n\}$. Denote by $\hat{\gamma}_n$ a curve 
given by the equation $z=x+i\hat{y}_n(x)$, where $\hat{y}_n(x)=y_n(x\sqrt n)/\sqrt n$. The function $S_n(z)$ 
is continuous up to the real axis and it establishes a~homeomorphism between 
the real axis and the curve $\hat{\gamma}_n$.

Note as well that the function $T_n(z)$ maps $\mathbb C^+$ conformally onto $\hat{D}_n^*$,  
where $\hat{D}_n^*:=\{z=x+iy,x,y\in\mathbb R:y>y_n^*(x\sqrt n)/\sqrt n\}$. Here $y_n^*(x)$ 
is defined in the same way as $y_n(x)$ if we change the measure $\mu$ by $\mu^*$. 
By definition of the measure $\mu^*$, we see that $y_n^*(x)\le y_n(x)$. Hence 
$S_n(\mathbb C^+)\subseteq T_n(\mathbb C^+)$. Denote by $\hat{\gamma}_n^*$ a curve 
given by the equation $z=x+i\hat{y}_n^*(x),\,x\in\mathbb R$, where $\hat{y}_n^*(x)=y_n^*(x\sqrt n)/\sqrt n$. 
The function $T_n(z)$ 
is continuous up to the real axis and it establishes a~homeomorphism between 
the real axis and the curve $\hat{\gamma}_n^*$. Moreover, if $x_1<x_2$, then $T_n(x_1)<T_n(x_2)$.

\begin{proposition}\label{th4bpro1a}
For $z\in S_n(\mathbb C^+)$,
\begin{equation}\notag
|T_n^{(-1)}(z)-S_n^{(-1)}(z)|\le \frac{\tau(\{|u|>\sqrt{n-1}/\pi\})}{\Im z},\quad n\ge n_1(\mu).
\end{equation} 
\end{proposition}
\begin{proof}
Using (\ref{th4bpro1.1}) and the formula
\begin{equation}\label{th4bpro3.1}
S_n^{(-1)}(z)=nz-\frac{n-1}{\sqrt n}F_{\mu}(z\sqrt n)=z-\frac{n-1}{\sqrt n}\int_{\mathbb R}
\frac{\tau(du)}{u-z\sqrt n},\quad z\in\mathbb C^+,
\end{equation} 
we have, taking into account that $S_n(\mathbb C^+)\subseteq T_n(\mathbb C^+)$,
\begin{equation}\notag
|S_n^{(-1)}(z)-T_n^{(-1)}(z)|=\frac{n-1}{\sqrt n}\Big|\int_{|u|>\sqrt{n-1}/\pi}\frac{\tau(du)}{u-z\sqrt n}\Big| \le
\frac{\tau(\{|u|>\sqrt{n-1}/\pi\})}{\Im z} 
\end{equation}
for $z\in S_n(\mathbb C^+)$, proving the proposition.
\end{proof}
\begin{proposition}\label{th4bpro1b}
For $x_1,x_2\in \hat{I}_n$, we have the estimate 
\begin{equation}\notag
|T_n(x_1)-T_n(x_2)|\le c(\mu)\frac{|x_1-x_2|+n^{-3/2}}
{\min_{j=1,2}\{\sqrt{4-(e_n(x_j-a_n))^2}\}}.
\end{equation}
 
\end{proposition}
\begin{proof}
By Theorem~\ref{th4a}, we have the following relation
\begin{align}\label{th4bpro1b.1}
T_n(x_1)-T_n(x_2)=M_n(x_1)-M_n(x_2)+\frac{r_{n1}(x_1)n^{-3/2}}{\sqrt{(e_n(x_1-a_n))^2-4}}
-\frac{r_{n1}(x_2)n^{-3/2}}{\sqrt{(e_n(x_2-a_n))^2-4}},
\end{align}
where $x_1,x_2\in \hat{I}_n$.
On the other hand it is easy to see that
\begin{align}\label{th4bpro1b.2}
&M_n(x_1)-M_n(x_2)=\frac{(1+b_n)(x_1-x_2)}2\notag\\ 
&+\frac 12\frac{(1-b_n)^2(x_1-x_2)(x_1+x_2-2a_n)}
{\sqrt{(1-b_n)^2(x_1-a_n)^2-4(1-d_n)}+\sqrt{(1-b_n)^2(x_2-a_n)^2-4(1-d_n)}}.
\end{align} 
Moreover, we have, for $x_1,x_2\in \hat{I}_n$,
\begin{align}
&\Big|\sqrt{(1-b_n)^2(x_1-a_n)^2-4(1-d_n)}
+\sqrt{(1-b_n)^2(x_2-a_n)^2-4(1-d_n)}\Big|\notag\\
&\ge\frac 12(|\sqrt{4(1-d_n)-(1-b_n)^2(x_1-a_n)^2}|+|\sqrt{4(1-d_n)-(1-b_n)^2(x_2-a_n)^2}|). \notag
\end{align} 
In view of this bound and (\ref{th4bpro1b.1}), (\ref{th4bpro1b.2}), we easily obtain the assertion of the proposition.
\end{proof}

\begin{proposition}\label{th4bpro1c}
For $x\in \hat{I}_n$, 
the following formula holds
\begin{align}\notag
T_n(x)=a_n+\frac 12\Big((1+b_n)(x-a_n)+(1-b_n)\sqrt{(x-a_n)^2-4/e_n^2}\Big)+\frac{r_{n5}(x)}
{\sqrt{(e_n(x-a_n))^2-4}},  
\end{align}
where $|r_{n5}(x)|\le c(\mu)n^{-3/2}$.
\end{proposition}
\begin{proof}
The proof immediately follows from Theorem~\ref{th4a}. 
\end{proof}

Let $h^*=n^{-6/5}$. Denote
\begin{align}
A_{n}^*:=T_n(x_{n1}^*),\,\, 
B_{n}^*:=T_n(x_{n2}^*),\quad\text{where}\quad x_{n1}^*:=a_n-2/e_n+h^*,\,\, x_{n2}^*:=a_n+2/e_n-h^*.\notag
\end{align}

From Proposition~\ref{th4bpro1c} we see that, for $|x-a_n|\le \frac {2}{e_n}-h^*$, 
\begin{align}
&\Im T_n(x)\ge \frac 12(1-b_n)\sqrt{4/e_n^2-(x-a_n)^2}-c(\mu)(h^*)^2\ge
\frac 12 \sqrt{h^*}-c(\mu)(h^*)^2\ge\frac 14\sqrt{h^*},\notag\\
&\Im T_n(x)\le \frac 12(1-b_n)\sqrt{4/e_n^2-(x-a_n)^2}+c(\mu)(h^*)^2\le 2.\notag
\end{align} 
Moreover we note as well that 
\begin{equation}\notag
|\Re T_n(x)|\le |a_n|+\frac 12(1-b_n)|x-a_n|+c(\mu)(h^*)^2. 
\end{equation}
Therefore 
\begin{equation}\label{th4.7}
\frac 14\sqrt{h^*}\le \Im A_n^*\le \sqrt{h^*},\quad
\frac 14\sqrt{h^*}\le \Im B_n^*\le \sqrt{h^*} 
\end{equation}
and
\begin{equation}\label{th4.7a}
|\Re A_n^*-a_n+1/e_n|+|\Re B_n^*-a_n-1/e_n|\le 2h^*.  
\end{equation}

Since $T_n(x),\,x\in\mathbb R$, is a homeomorphism between the real axis and the curve $\hat{\gamma}_n^*$ 
and $\Re T_n(x_1)<\Re T_n(x_2)$ for $x_1<x_2$,
we see that the function $\hat{y}_n^*(x)$ satisfies the inequality 
\begin{equation}\label{th4.7**}
\frac 14\sqrt{h^*}\le\hat{y}_n^*(x)\le 2\quad\text{for}\quad \Re A_{n}^*\le x\le\Re B_{n}^*.
\end{equation}

We saw before that $\hat{y}_n(x)\ge \hat{y}_n^*(x),\,x\in\mathbb R$. In the following proposition we evaluate
the distance between these functions. Denote $\eta_n:=\tau(\{|u|>\sqrt{n-1}/\pi\})(h^*)^{-3/2}$. It is easy to see that
$\eta_n\le c(\mu)n^{-6/5}$.
\begin{proposition}\label{th4bpro3}
For $\Re A_{n}^*\le x\le\Re B_{n}^*$, we have the estimate 
\begin{equation}\notag
0\le \hat{y}_n(x)-\hat{y}_n^*(x)\le c(\mu)\eta_n.
\end{equation}
\end{proposition}
\begin{proof}
For $\Re A_{n}^*\le x\le\Re B_{n}^*$ and $z=x+i(\hat{y}_n^*(x)+v)$ with $0\le v\le 3$, we have the following 
lower bound, using (\ref{th4.7a}) and (\ref{th4.7**}),
\begin{align}\label{th4bpro3.2}
&\int_{|u|\le \sqrt{n-1}/\pi}\frac{\tau(du)}{(u-x\sqrt n)^2+n(\hat{y}_n^*(x))^2}-
\int_{|u|\le \sqrt{n-1}/\pi}\frac{\tau(du)}{(u-x\sqrt n)^2+n(\hat{y}_n^*(x)+v)^2}\notag\\
&\ge \int_{|u|\le \sqrt{n-1}/\pi}\frac{nv(2\hat{y}_n^*(x)+v)\,\tau(du)}{((u-x\sqrt n)^2+n(\hat{y}_n^*(x)+v)^2)^2}
\ge\frac {c(\mu)\hat{y}_n^*(x)v}n
\end{align}
and the upper bound
\begin{equation}\label{th4bpro3.3}
\int_{|u|>\sqrt{n-1}/\pi}\frac{\tau(du)}{(u-\Re x\sqrt n)^2+n(\hat{y}_n^*(x)+v)^2}\le
\frac{\tau(\{|u|>\sqrt{n-1}/\pi\})}{n(\hat{y}_n^*(x))^2}.
\end{equation}
Applying (\ref{th4bpro3.2}) and (\ref{th4bpro3.3}) to the formula (\ref{th4bpro3.1}) we see
that, for the $x$ considered above, $\Im S_n^{(-1)}(x+i(\hat{y}_n^*(x)+v_0)\ge 0$, where 
$v_0=c(\mu)\tau(\{|u|>\sqrt{n-1}/3\})/(\hat{y}_n^*(x))^3$ with some constant $c(\mu)$.
By (\ref{th4.7**}), we arrive at the~assertion of the proposition.
\end{proof}

Consider the points $A_{n}:=A_{n}^*+ih_{-}$ and $B_{n}:=B_{n}^*+ih_{+}$,
where $h_{\pm}\ge 0$ is chosen in the way such that $A_{n}\in\hat{\gamma}_n$ and 
$B_{n}\in\hat{\gamma}_n$.
Denote $x_{n1}:=S_n^{(-1)}(A_{n})$ and $x_{n2}:=S_n^{(-1)}(B_{n})$.

\begin{proposition}\label{th4bpro4}
The following inequalities hold
\begin{equation}\notag
|x_{n1}^*-x_{n1}|+|x_{n2}^*-x_{n2}|\le c(\mu)\eta_n.
\end{equation}
\end{proposition}
\begin{proof}
Write
\begin{equation}\notag
|x_{n1}^*-x_{n1}|\le |T_n^{(-1)}(A_{n}^*)-T_n^{(-1)}(A_{n})|
+|T_n^{(-1)}(A_{n})-S_n^{(-1)}(A_{n})|.
\end{equation}
By Propositions~\ref{th4bpro1} and \ref{th4bpro3} we have the upper bound
\begin{equation}\notag
|T_n^{(-1)}(A_{n}^*)-T_n^{(-1)}(A_{n})|=|T_n^{(-1)}(A_{n}^*)-T_n^{(-1)}(A_{n}^*+ih_{-})|\le 
c(\mu)\eta_n.
\end{equation}
Moreover, by Proposition~\ref{th4bpro1a} and the estimate (\ref{th4.7}), we obtain
\begin{equation}\notag
|T_n^{(-1)}(A_{n})-S_n^{(-1)}(A_{n})|\le \tau(\{|u|>\sqrt{n-1}/\pi\}/{\Im A_n}\le
4\tau(\{|u|>\sqrt{n-1}/\pi\}/\sqrt{h^*}.
\end{equation}
Hence
\begin{equation}\notag
|x_{n1}^*-x_{n1}|\le c(\mu)\eta_n.
\end{equation}
The same estimate holds for $|x_{n2}^*-x_{n2}|$. Thus, the proposition is proved. 
\end{proof}

Now we can complete the proof of Theorem~\ref{th4b}.
Consider the interval $[\tilde{x}_{n1},\tilde{x}_{n2}]$, where $\tilde{x}_{n1}:=\max\{x_{n1}^*,x_{n1}\}$ and
$\tilde{x}_{n2}:=\min\{x_{n2}^*,x_{n2}\}$. Note that, by Proposition~\ref{th4bpro4}, 
$|\tilde{x}_{n1}-a_n+2/e_n|\le c(\mu)n^{-6/5}$ and $|\tilde{x}_{n2}-a_n-2/e_n|\le c(\mu)n^{-6/5}$.

Let $x\in[\tilde{x}_{n1},\tilde{x}_{n2}]$. Since $S_n(x)\in\hat{\gamma}_n$ and $T_n(x)\in\hat{\gamma}_n^*$,
by Proposition~\ref{th4bpro3}, there exists 
\begin{equation}\label{th4b.1}
h(x)\in[0,c(\mu)\eta_n]
\end{equation}
such that $S_n(x)-ih(x)\in\hat{\gamma}_n^*$. By Proposition~\ref{th4bpro1},
\begin{equation}\notag
|T_n^{(-1)}(S_n(x))-T_n^{(-1)}(S_n(x)-ih(x))|\le c(\mu)h(x)\le
c(\mu)\eta_n  
\end{equation}
and, by Proposition~\ref{th4bpro1a} and the lower bound in (\ref{th4.7**}),
\begin{equation}\notag
|T_n^{(-1)}(S_n(x))-S_n^{(-1)}(S_n(x))|\le\frac{\tau(\{|u|>\sqrt{n-1}/\pi\})}{\Im S_n(x)}
\le 4\frac{\tau(\{|u|>\sqrt{n-1}/\pi\})}{\sqrt{h^*}}. 
\end{equation}
Using the last two bounds we obtain
\begin{align}\label{th4b.2}
&|S_n^{(-1)}(S_n(x))-T_n^{(-1)}(S_n(x)-ih(x))|\notag\\
&\le |T_n^{(-1)}(S_n(x))-S_n^{(-1)}(S_n(x))|
+|T_n^{(-1)}(S_n(x))-T_n^{(-1)}(S_n(x)-ih(x))|\le c(\mu)\eta_n.\notag\\
\end{align}
Denote $\tilde{x}:=T_n^{(-1)}(S_n(x)-ih(x))$. This point lies in $[\tilde{x}_{n1},\tilde{x}_{n2}]$ 
and, by (\ref{th4b.2}), $|x-\tilde{x}|\le c(\mu)\eta_n$. It remains to note that
\begin{equation}\notag
|T_n(x)-S_n(x)|\le |T_n(x)-T_n(\tilde{x})|+|T_n(\tilde{x})-S_n(x)|=|T_n(x)-T_n(\tilde{x})|+|h(x)|.
\end{equation}
It follows from Proposition~\ref{th4bpro1b} and the relation (\ref{th4b.1}) that
\begin{equation}\notag
|T_n(x)-S_n(x)|\le c(\mu)\frac{\eta_n+n^{-3/2}}{\sqrt{(2/e_n)^2-(x-a_n)^2}},\quad 
x\in[\tilde{x}_{n1},\tilde{x}_{n2}]. 
\end{equation} 
The statement of the theorem immediately follows from this inequality.

\section{Asymptotic expansion of $\int_{\mathbb R} |p_n(x)-p_w(x)|\,dx$} 

In this section we shall prove Corollary~\ref{corth7.1}. Indeed, by the estimate (\ref{asden2}), 
we have, for $n\ge n_1$, 
\begin{align}\label{exp1.1}
\int_{\mathbb R} |p_n(x)-p_w(x)|\,dx=\int_{I_n} |p_n(x)-p_w(x)|\,dx+\theta\frac {c(\mu)}{n^{6/5}}.
\end{align}
First let us assume that $m_3\ne 0$. Using Theorems~\ref{th5},~\ref{th4c},
we easily conclude that
\begin{align}\notag
&\int_{I_n} |p_n(x)-p_w(x)|\,dx=\int_{I_n} |p_n^*(x)-p_w(x)|\,dx+\theta\int_{I_n} |p_n^*(x)-p_w(x)
-\tilde{\rho}(x)|\,dx\notag\\ 
&=\int_{[-2,2]} |(1-a_nx)p_w(x)-p_w(x+a_n)|\,dx+\theta c(\mu)(|a_n|^{3/2}+n^{-1})\notag\\
&=\frac{|a_n|}{2\pi}\int_{[-2,2]} |x|\frac{|3-x^2|}{\sqrt{4-x^2}}\,dx+\theta c(\mu)(|a_n|^{3/2}+n^{-1})=
\frac{2|a_n|}{\pi}+c(\mu)(|a_n|^{3/2}+n^{-1}).\notag
\end{align}
Applying this relation to (\ref{exp1.1}) we get the expansion (\ref{2.9}).

Let $m_3=0$. By Theorems~\ref{th5} and \ref{th4c},
we get
\begin{align}\notag
&\qquad\qquad\qquad\qquad\int_{I_n} |p_n(x)-p_w(x)|\,dx\notag\\
&=\int_{[-2,2]}\Big|p_w(x)-p_w(\frac x{e_n})+\Big(\frac 12 d_n-\frac 1n-\Big(b_n
-\frac 1n\Big)\frac{x^2}{e_n^2}\Big)p_w(x)\Big|\,\frac{dx}{e_n}+\theta\frac {c(\mu)}{n^{6/5}}\notag\\
&=\frac 1{2\pi}\int_{[-2,2]}\frac{|2\Big(d_n-\frac 2n\Big)-\Big(4\Big(b_n
-\frac 1n\Big)+\frac 12 d_n-\frac 1n-b_n+\frac 12d_n\Big)x^2+\Big(b_n-\frac 1n\Big)x^4|}{\sqrt{4-x^2}}\,dx\notag\\
&+\theta\frac {c(\mu)}{n^{6/5}}=\frac{|m_4-2|}{2\pi n}\int_{[-2,2]}\frac{|2-4x^2+x^2|}{\sqrt{4-x^2}}\,dx
+\theta\frac {c(\mu)}{n^{6/5}}=\frac{2|m_4-2|}{\pi n}+\theta\frac {c(\mu)}{n^{6/5}}.
\end{align}
Taking into account (\ref{exp1.1}) we obtain (\ref{2.10}).

\section{Asymptotic expansion of the free entropy and the free Fisher information} 

In this section we prove Corollaries~\ref{corth7.2} and \ref{corth7.3}. First we find an asymptotic expansion of 
the logarithmic energy $E(\mu_n)$ of the measure $\mu_n$. Recall that (see~\cite{HiPe:2000})
\begin{align}\label{exp.1}
-E(\mu_n)&=\int\int_{\mathbb R^2}\log|x-y|\,\mu_n(dx)\mu_n(dy)=I_1(\mu_n)+I_2(\mu_n):=\notag\\
&=\int\int_{I_n\times I_n}\log|x-y|\,\mu_n(dx)\mu_n(dy)+
\int\int_{\mathbb R^2\setminus (I_n\times I_n)}\log|x-y|\,\mu_n(dx)\mu_n(dy).
\end{align}
Using (\ref{asden3}) and (\ref{asden2}),   
we conclude that
\begin{align}\label{exp.2}
|I_2(\mu_n)|&\le 4\int_{\mathbb R\setminus I_n}p_n(x)\,\int_{\mathbb R}|\log|x-y||\,p_n(y)\,dy\,dx\le 
c(\mu)\int_{\mathbb R\setminus I_n}p_n(x)\,dx\notag\\
&\le c(\mu)n^{-6/5}. 
\end{align}
By (\ref{asden}), we have 
\begin{align}\label{exp.3}
I_1(\mu_n)&= \int\int_{I_n\times I_n}\log|x-y|\,v_n(x-a_n)v_n(y-a_n)\,dx\,dy
+c(\mu)\theta n^{-6/5}\notag\\
&=\int\int_{\mathbb R^2}\log|x-y|\,v_n(x)v_n(y)\,dx\,dy
+c(\mu)\theta n^{-6/5}.
\end{align}
Recalling the definition of $v_n(x)$ we see that
\begin{align}\label{exp.4}
&\int\int_{\mathbb R^2}\log|x-y|\,v_n(x)v_n(y)\,dx\,dy=\tilde{I}_1(v_n)+\tilde{I}_2(v_n)+\tilde{I}_3(v_n)
+\tilde{I}_4(v_n)+r_n\notag\\
&:=\Big(1+\frac 12 d_n-a_n^2-\frac 1{n}\Big)^2 \int\int_{\mathbb R^2}\log|x-y|\,
p_w(e_nx)\,p_w(e_ny)\,dx\,dy\notag\\
&-2\Big(1+\frac 12 d_n-a_n^2-\frac 1{n}\Big)a_n\int\int_{\mathbb R^2}x\log|x-y|\,
p_w(e_nx)\,p_w(e_ny)\,dx\,dy\notag\\
&+a_n^2\int\int_{\mathbb R^2}xy\log|x-y|\,p_w(e_nx)\,
p_w(e_ny)\,dx\,dy\notag\\
&-2\Big(b_n-a_n^2-\frac 1{n}\Big)\int\int_{\mathbb R^2}x^2\log|x-y|\,
p_w(e_nx)\,p_w(e_ny)\,dx\,dy
+c(\mu)\theta n^{-6/5}. 
\end{align}

In view of $E(\mu_w)=1/4$ and $e_n^{-2}=1-d_n+2b_n+c(\mu)\theta n^{-2},\,
\log e_n=\frac{d_n}2-b_n+c(\mu)\theta n^{-2}$,
note that
\begin{align}\label{exp.5}
\tilde{I}_1(v_n)&=\Big(1+\frac 12 d_n-a_n^2-\frac 1{n}\Big)^2 e_n^{-2}
\Big(\int\int_{\mathbb R^2}\log|x-y|
p_w(x)p_w(y)\,dx\,dy-\log e_n\Big)\notag\\
&=-E(\mu_w)+\frac{a_n^2}2+c(\mu)\theta n^{-2}. 
\end{align}
Since the function $\int_{\mathbb R}\log|x-y|p_w(y)\,dy$ is even, we see that
$\tilde{I}_2(v_n)=0$. In order to calculate $\tilde{I}_3(v_n)$ we easily deduce that
\begin{equation}\notag
\int_{-2}^2 up_w(u)\log|x-u|\,du=-x+x^3/6,\quad x\in[-2,2]. 
\end{equation}
Therefore we obtain
\begin{equation}\label{exp.6}
\tilde{I}_3(v_n)=-a_n^2\int_{-2}^2xp_w(x)(x-x^3/6)\,dx+c(\mu)\theta n^{-2}=-\frac 23a_n^2+c(\mu)\theta n^{-2}. 
\end{equation}

Using the following well-known formula (see~\cite{HiPe:2000}, p. 197)
\begin{equation}\notag
\int_{-2}^2 p_w(u)\log|x-u|\,du=\frac{x^2}4-\frac 12,\quad x\in[-2,2],
\end{equation}
we deduce
\begin{equation}\notag
\int\int_{\mathbb R^2}x^2\log|x-y|\,
p_w(x)\,p_w(y)\,dx\,dy=\int_{-2}^2x^2\Big(\frac{x^2}4-\frac 12\Big)p_w(x)\,dx=0. 
\end{equation}
Therefore
\begin{equation}\label{exp.7}
\tilde{I}_4(v_n)=c(\mu)\theta n^{-2}. 
\end{equation}

By (\ref{exp.3})--(\ref{exp.7}), we arrive at the formula
\begin{equation}\notag
I_1(\mu_n)=-E(\mu_w)-\frac 16a_n^2+c(\mu)\theta n^{-3/2}. 
\end{equation}
Finally, in view of (\ref{exp.1})--(\ref{exp.2}), we get
\begin{equation}\notag
-E(\mu_n)=-E(\mu_w)-\frac 16a_n^2+c(\mu)\theta n^{-3/2}  
\end{equation}
and we obtain the assertion of Corollary~\ref{corth7.2}.

Now let us prove Corollary~\ref{corth7.3}. We shall show that the free Fisher information of the measure $\mu_n$ 
has the form (\ref{2.11}).
Denote 
\begin{align}\label{exp.9}
\Phi(\mu_n)=\Phi_1(\mu_n)+\Phi_2(\mu_n):=\frac{4\pi^2}3\int_{I_n}p_n(x)^3\,dx+
\frac{4\pi^2}3\int_{\mathbb R\setminus I_n}p_n(x)^3\,dx.
\end{align}
As before we see that, by (\ref{asden2}),
\begin{equation}\label{exp.10}
\Phi_2(\mu_n)\le c(\mu)\int_{\mathbb R\setminus I_n}p_n(x)\,dx\le c(\mu)n^{-6/5}.
\end{equation}
On the other hand, by (\ref{loc.6}), we have 
\begin{equation}\label{exp.11}
\Phi_1(\mu_n)=\frac{4\pi^2}3\int_{\mathbb R}v_n(x)^3\,dx+c(\mu)\theta n^{-6/5}. 
\end{equation}
It is easy to see that the integral on the right hand-side of (\ref{exp.11}) is equal to
\begin{align}
\Big(1&+3\Big(\frac 12 d_n-a_n^2-\frac 1n\Big)\Big)e_n^{-1}\int_{\mathbb R}p_w(x)^3\,dx
-3\Big(b_n-a_n^2-\frac 1n\Big)e_n^{-3}\int_{\mathbb R}x^2p_w(x)^3\,dx \notag\\
&+3a_n^2e_n^{-3}\int_{\mathbb R}x^2p_w(x)^3\,dx+c(\mu)\theta n^{-3/2}=\frac 3{4\pi^2}(1+a_n^2)
+c(\mu)\theta n^{-6/5}.\notag
\end{align}
Therefore we finally have by (\ref{exp.9})--(\ref{exp.11})
\begin{equation}\notag
\Phi(\mu_n)=1+a_n^2+c(\mu)\theta n^{-3/2}=\Phi(\mu_w)+a_n^2
+c(\mu)\theta n^{-6/5}. 
\end{equation}
Thus, Corollary~\ref{corth7.3} is proved.

\end{document}